\documentclass[12pt]{article}
\usepackage{amsmath,amssymb,amsthm,graphicx}
\usepackage{enumitem}
\usepackage{hyperref}
\usepackage{fullpage}

\newtheorem{theorem}{Theorem}[section]
\newtheorem{lemma}[theorem]{Lemma}
\newtheorem{Obs}[theorem]{Observation}

\newtheorem{conjecture}[theorem]{Conjecture}

\begin{document}

\title{Almost All Regular Graphs are Normal}

\author{Seyed Saeed Changiz Rezaei\thanks{Supported in part by a MITACS Accelerate postdoctoral grant.}\\[1mm]
\and
Seyyed Aliasghar Hosseini\\[1mm]
\and
Bojan Mohar\thanks{Supported in part by an NSERC Discovery Grant (Canada),
   by the Canada Research Chair program, and by the
    Research Grant P1--0297 of ARRS (Slovenia).}~\thanks{On leave from:
    IMFM \& FMF, Department of Mathematics, University of Ljubljana, Ljubljana,
    Slovenia.}\\[1mm]
  Department of Mathematics\\
  Simon Fraser University\\
  Burnaby, BC, Canada\\
}

\date{\today}

\maketitle

\begin{abstract}
 In 1999, De Simone and K\"{o}rner conjectured that every graph without 
induced $C_5,C_7,\overline{C}_7$ contains a clique cover $\mathcal C$ 
and a stable set cover $\mathcal I$ such that every clique in $\mathcal 
C$ and every stable set in $\mathcal I$ have a vertex in common. This 
conjecture has roots in information theory and became known as the 
Normal Graph Conjecture. Here we prove that all graphs of bounded 
maximum degree and sufficiently large odd girth (linear in the maximum 
degree) are normal. This implies that for every fixed $d$,  random 
$d$-regular graphs are a.a.s.\ normal.
\end{abstract}

\section{Introduction}
\label{sec:basics}

A graph $G$ is said to be \emph{normal} if it contains a set $\mathcal C$ of cliques and a set $\mathcal I$ of stable sets with the following properties:
\begin{itemize}
  \item[(1)] $\mathcal C$ is a cover of $G$, i.e., every vertex in $G$ belongs to one of the cliques in $\mathcal C$;
  \item[(2)] $\mathcal I$ is a cover of $G$, i.e., every vertex in $G$ belongs to one of the stable sets in $\mathcal I$;
  \item[(3)] Every clique in $\mathcal C$ and every stable set in $\mathcal I$ have a vertex in common.
\end{itemize}

Clearly, a graph is normal if and only if its complement is normal.
This property is reminiscent on the notion of perfect graphs. Namely, normality is one of the basic properties that every perfect graph satisfies. Of course, normality is much weaker condition since every odd cycle of length at least 9 is normal.

The importance of normality of graphs lies in its close relationship to the notion of graph entropy, one of central concepts in information theory; see Csisz{\'a}r and K\"orner \cite{CsKor11} or \cite{CsKoLoMaSi90,KorLongo73,Korner73}.

A set $\mathcal C$ of edges of a graph $G$ is a \emph{star cover} of $G$ if every vertex of positive degree in $G$ is incident with an edge in $\mathcal C$ and each component formed from the edges in $\mathcal C$ is a star (a graph isomorphic to $K_{1,t}$ for some $t\ge1$). In the definition of normality, one may ask that the clique cover $\mathcal C$ is minimal. Note that a minimal clique cover in a triangle-free graph is the same as a star cover. A star cover $\mathcal C$ of a graph $G$ is \emph{nice} if every odd cycle in $G$ contains at least 3 vertices whose incident edges in the cycle are either both or none in $\mathcal C$.

For triangle-free graphs, De Simone and K\"orner \cite{DeSimoneKorner1999} proved the following relationship between normality and existence of nice star covers.

\begin{theorem}
\label{thm:DeSimoneKorner}
A triangle-free graph is normal if and only if it has a nice star cover.
\end{theorem}

Let $Q$ be a cycle of a graph $G$ and $\mathcal C$ be a star cover of $G$. Then a vertex $v$ of $Q$ is a \emph{good vertex} (with respect to $\mathcal C$) if the two edges of $Q$ incident with $v$ are either both in $\mathcal C$ or none is in $\mathcal C$. We define when a degree-2 vertex on a path is a \emph{good vertex} in the same way.
Let $h$ be the number of components of $Q\cap \mathcal{C}$. Note that the number of good vertices in the cycle $Q$ is equal to $|Q| - 2h$.
Hence, we have the following observation.

\begin{Obs}
\label{Obs:Obsr1}
Let $Q$ be an odd cycle. Then the number of good vertices of $Q$ is odd.
\end{Obs}

This observation shows that a star cover $\mathcal C$ is nice if and only if every odd cycle $Q$ of $G$ has at least two good vertices.

In their inspiring work, De Simone and K\"orner \cite{DeSimoneKorner1999} proposed the following conjecture:

\begin{conjecture}
\label{conj:Korner}
A graph is normal if it contains no induced cycles of length 5 or 7 and does not contain the complement of $C_7$ as an induced subgraph.
\end{conjecture}

This conjecture has been verified for a very small class of graphs, including line graphs of cubic graphs \cite{Patakfalvi2008}, minimal asteroidal triple graphs \cite{Talmaciu2013}, and a specific class of circulants \cite{Wagler2007}. See also \cite{Talmaciu2009} and \cite{Wagler2008}. We proved this conjecture for triangle-free subcubic graphs in \cite{CubicNormal2014}. The main result of this paper is Theorem \ref{thm:last}  which states that random $d$-regular graphs are a.a.s\ normal.

To prove the Theorem \ref{thm:last}, we first show that all graphs with bounded maximum degree and sufficiently large odd girth which is linear in maximum degree are normal. The following theorem states this result: Every $C_4$-free graph $G$ with maximum degree $k\ge 3$ and odd girth at least\/
$16k-19$ is normal. See Theorem \ref{thm:MainGeneral}.

The proof technique of \cite{CubicNormal2014} would enable us to prove Theorem \ref{thm:MainGeneral} without excluding $4$-cycles. However, since our main motivation is Theorem \ref{thm:last}, it does not seem to be worth of the additional efforts.

\section{Graphs with maximum degree $k$}

Here we extend the method that was used in \cite{CubicNormal2014} for cubic graphs to graphs with vertices of degrees more
than 3. The proof is by induction on $\Delta(G)$ and
gives a bound on the odd girth that is linear in $\Delta(G)$. We
will sustain of trying to optimize this bound since it is too
far from the conjectured value in Conjecture \ref{conj:Korner}.

Let us define a star cover $\mathcal C$ of a graph $G$ by using
the following procedure. We assume that $G$ has no isolated vertices.

\begin{enumerate}
\item[(1)] Let $k=\Delta(G)$, $G_k=G$, and let $s=k$.

\item[(2)] Let $F_s$ be a maximum set of vertex-disjoint $s$-stars
such that $G_s'=G_s\backslash V(F_s)$ has no isolated vertices
apart from those that are already present in $G_s$. (When $s=k$,
there are none, but they may arise later when this step is repeated
with $s<k$.)

\item[(3)] Let $U_s'$ be the set of vertices whose degree in
$G_s'$ is equal to 1 and whose neighbor has degree $s$ in $G_s'$.
Let $U_s''$ be the neighbors in $G_s'$ of vertices in $U_s'$ and
let $U_s = U_s'\cup U_s''$. Add all $U_s'-U_s''$ edges to the
cover and let $G_{s-1}=G_s'\backslash U_s$. Note that
$\Delta(G_{s-1}) \le s-1$ and that the last step may give rise to
some isolated vertices in $G_{s-1}$.

\item[(4)] If $s>3$ then decrease $s$ by one and go to (2).

\item[(5)] Note that the subgraph $G_2$ consists of paths and
cycles, possibly including some isolated vertices. Let $V_2$ be
the set of isolated vertices in $G_2$. For each vertex in $V_2$,
choose one of its neighbors in $U''_s$ where $s$ is as small
as possible. Add the edge joining them into $\mathcal{C}$. If we
have a pattern of edges in $\mathcal C$ like the one shown at the top of Figure \ref{fig:5}, we change it to the cover shown at the bottom of Figure \ref{fig:5}.

\item[(6)] Note that the subgraph $G'' = G_2\backslash V_2$
consists of non-trivial paths and cycles. Cover the vertices of
each path or cycle in $G''$ with 2-stars and at most 2 single
edges.
\end{enumerate}

\begin{figure}
  \centering
  \includegraphics{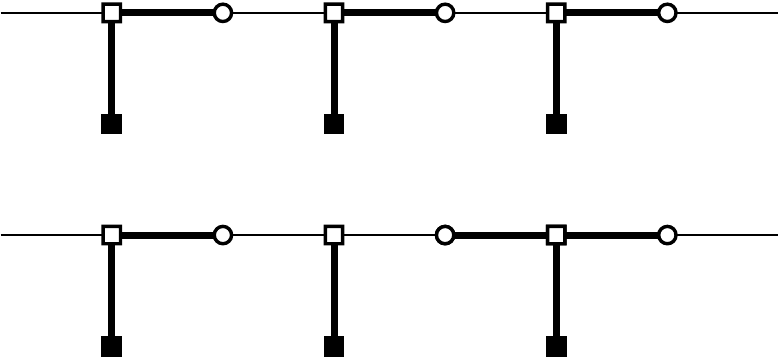}
  \caption{Changing the cover of vertices in $V_2$ from the pattern shown above to the one below. Thicker edges are in $\mathcal C$, the circle vertices are in $V_2$, the full square vertices are all in $U_s'$ and the empty square vertices are all in $U_s''$ for the same value of~$s$.}
  \label{fig:5}
\end{figure}

Using the above construction of a star cover $\mathcal C$, we can
prove the following.

\begin{theorem}
\label{thm:MainGeneral}
Every $C_4$-free graph $G$ with maximum degree $k\ge 3$ and odd girth at least\/
$16k-19$ is normal.
\end{theorem}

Let us consider the construction (1)--(6) of the star cover
$\mathcal{C}$ described above. Furthermore, note that if $Q$ is a cycle in $G$ and $v\in V(Q)\cap F_k$ where $k=\Delta(G)$, then either $v$ is a good vertex of $Q$, or there is an edge
$vu\in E(Q)\cap E(F_k)$ such that $u$ is a good vertex in $Q$.

\begin{lemma}
\label{lem:c-alt} Let $P= v_0v_1v_2 \ldots v_8v_9$ be a
$\mathcal{C}$-alternating path such that $v_0v_1\in \mathcal{C}$.
Let $T_1(s)=\{v_iv_{i+1} : v_iv_{i+1}\in S(v_{i+1})\subseteq
F_s\}$, $T_2(s)=\{v_iv_{i+1} : v_i\in U_s'' ,  v_{i+1}\in U_s'\}$
and $T_3(s)=\{v_iv_{i+1}: v_i\in U_s'' ,  v_{i+1}\in V_2\}$. 
If\/ $v_0v_1\in T_1(s)\cup T_2(s)\cup T_3(s)$, then one of the edges
among $v_2v_3, v_4v_5, v_6v_7$ or $v_8v_9$ is in $T_1(t)\cup
T_2(t)\cup T_3(t)$ for $t>s$.
\end{lemma}

\begin{proof}
The proof is split into three cases depending on whether $v_0v_1$ is in $T_i(s)$ for $i=1, 2,$ or 3.

\textbf{Case 1.} First assume that $v_0v_1\in T_1(s)$.
Since $P$ is a $\mathcal{C}$-alternating path, $v_1v_2\notin
\mathcal{C}$. Therefore $v_2$ is not in the star in $F_s$ centered at
$v_1$. This implies that $v_2$ has been selected in some step
before $s$, because at each step we only select maximum stars
in the remaining graph. This means that $v_2\in F_t\cup U_t''$ for some
$t>s$. Suppose that $v_2\in F_t$. Since the edge $v_2v_1\in E(G_t)$ is not in the star of $F_t$ containing $v_2$, the star is centered at $v_3$. This shows that $v_2v_3\in T_1(t)$. For the other possibility, suppose that $v_2\in U_t''$.
Again, $v_2v_3\in \mathcal{C}$ implies that $v_3\in U_t'$ or $v_3\in
V_2$. Therefore $v_2v_3\in T_2(t)\cup T_3(t)$.

\textbf{Case 2.} Assume that $v_0v_1\in T_2(s)$. Since $v_1\in
U_s'$, $v_1$ was a vertex of degree one in $G'_s$. Therefore
$v_2\in F_s$ or $v_2\in F_t$ or $v_2\in U_t''$ for some $t>s$. Using
the same argument as in case 1, $v_2v_3$ or $v_4v_5\in T_1(t)\cup
T_2(t)\cup T_3(t)$ for $t>s$.

\textbf{Case 3.} Assume that $v_0v_1\in T_3(s)$. Since $v_1\in
V_2$, by step 5 of the construction of $\mathcal{C}$, $v_2$ has been selected at the same step as $v_0$ or at a step
before $v_0$. Therefore, $v_2\in U_t''\cup F_t$
for some $t\geq s$. If $v_2\in U_t''\cup F_t$, with $t>s$, or to $F_s$, then we are done by
cases 1 and 2, applied to the path $v_2v_3v_4\ldots$. Assume finally that $v_2\in U_s''$. Then
$v_3\in U_s'$ or $v_3\in V_2$. If $v_3\in U_s'$, case 2 shows that
$v_4v_5$ or $v_5v_6\in T_1(t)\cup T_2(t)\cup T_3(t)$ for some $t>s$.
Thus we may assume that $v_3\in V_2$. Again since $v_3\in V_2$, by step 5 of
the construction, $v_4$ has been selected at the same step or before $v_2$.
Therefore $v_4\in U_s''$ or $v_4\in F_s$ or $v_4\in U_t''\cup F_t$ for some $t>s$. If
$v_4\in U_t''\cup F_t\cap F_s$, we are done. Assume that $v_4\in U_s''$.
Step 5 of the construction indicates that $v_5\notin V_2$. Therefore
$v_5\in U_s'$. Case 2 shows that $v_6v_7$ or $v_8v_9\in T_1(t)\cup
T_2(t)\cup T_3(t)$ for $t>s$.
\end{proof}

\begin{lemma}
\label{lem:longest} If $G$ has maximum degree $k\ge
3$, any $\mathcal{C}$-alternating path has at most $16k-24$ vertices.
\end{lemma}

\begin{proof}
Consider a $\mathcal{C}$-alternating path $P=u_0u_1u_2\ldots$. If an edge $u_iu_{i+1}\in \mathcal{C}$ is in $ T_1(s)\cup T_2(s)\cup T_3(s)$, then Lemma \ref{lem:c-alt} shows that (at least) every 8 vertices we reach an edge in $T_1(t)\cup T_2(t)\cup T_3(t)$ for some $t>s$. By the observation made before Lemma \ref{lem:c-alt}, once the path hits $F_k$, it can not continue and the vertices of $F_k$ will be the end of the path. Thus the path continues from $u_iu_{i+1}$ with at most $8(k-2)-2$ other vertices. At most 8 edges of $G''$ can be part of a $\mathcal{C}$-alternating path (step 6 of the construction). Thus $P$ is composed of at most two subpaths with $\leq 8(k-2)$ vertices each, plus at most 8 vertices in $G''$. The order of $P$ is therefore at most $16k-24$.
\end{proof}

\begin{proof}[Proof of Theorem~\ref{thm:MainGeneral}]
We assume that $\mathcal C$ has been constructed according to the
process described above. Let $Q$ be a cycle of odd length at least
$16k-19$. By Observation \ref{Obs:Obsr1}, $Q$ has a good vertex
$w$ and it suffices to prove that $Q$ has another good vertex. If
$w$ is the only good vertex of $Q$, deleting $w$ (and maybe its
neighbors in $Q$) from $Q$, gives us a $\mathcal{C}$-alternating
path of length $\ge 16k-22$. Lemma \ref{lem:longest} shows that
there is no alternating path of length $>16k-24$ in $G$. Therefore
$Q$ has another good vertex and this completes the proof.
\end{proof}

\section{Random Regular Graphs}
In this section, we show that all random regular graphs are a.a.s normal. Let $\mathcal{G}_{n,d}$ be the uniform probability space of $d$-regular graphs on the $n$ vertices $\{1,2,\cdots,n\}$ with $dn$ even. In other words, sampling from $\mathcal{G}_{n,d}$ is equivalent to picking such a graph uniformly at random (u.a.r.). Bollob\'{a}s proposed the \emph{configuration model} to generate the uniform probability space $\mathcal G_{n,d}$. The model is described as follows. Letting $W = \{1,\cdots,n\}\times\{1,\cdots, d\}$, a \emph{configuration} $P$ is a partition of $W$ into $\frac{dn}{2}$ pairs. The resulting pairs are called the \emph{edges} of the configuration and the points in $W$ are called \emph{half-edges}. By projecting the set $W$ to the set $\{1,\cdots,n\}$, each configuration $P$ projects to a $d$-regular multigraph $G(P)$ with vertex set $\{v_1,\cdots,v_n\}$. Furthermore, a pair $(x, y)$ in $P$ corresponds to an edge $(v_i, v_j )$ of $G(P )$ where $x = (i,k)$ for some $i\in\{1,\cdots,n\}$ and for some $k\in\{1,\cdots,d\}$, and $y = (j,l)$ for some  $j\in\{1,\cdots,n\}$ and some $l\in\{1,\cdots,d\}$. Since we are interested in simple graphs, each graph in the uniform probability space of $d$-regular graphs, i.e.,  $\mathcal{G}_{n,d}$, corresponds to precisely $d!^n$ configurations. Thus, by taking the projection of a random configuration and conditioning on it being a simple graph, we obtain a random $d$-regular graph on $\{1,\cdots,n\}$ with the uniform distribution over all such graphs (see \cite{janson2000random, wormald1999models}). 

\begin{figure}
  \centering
  \includegraphics{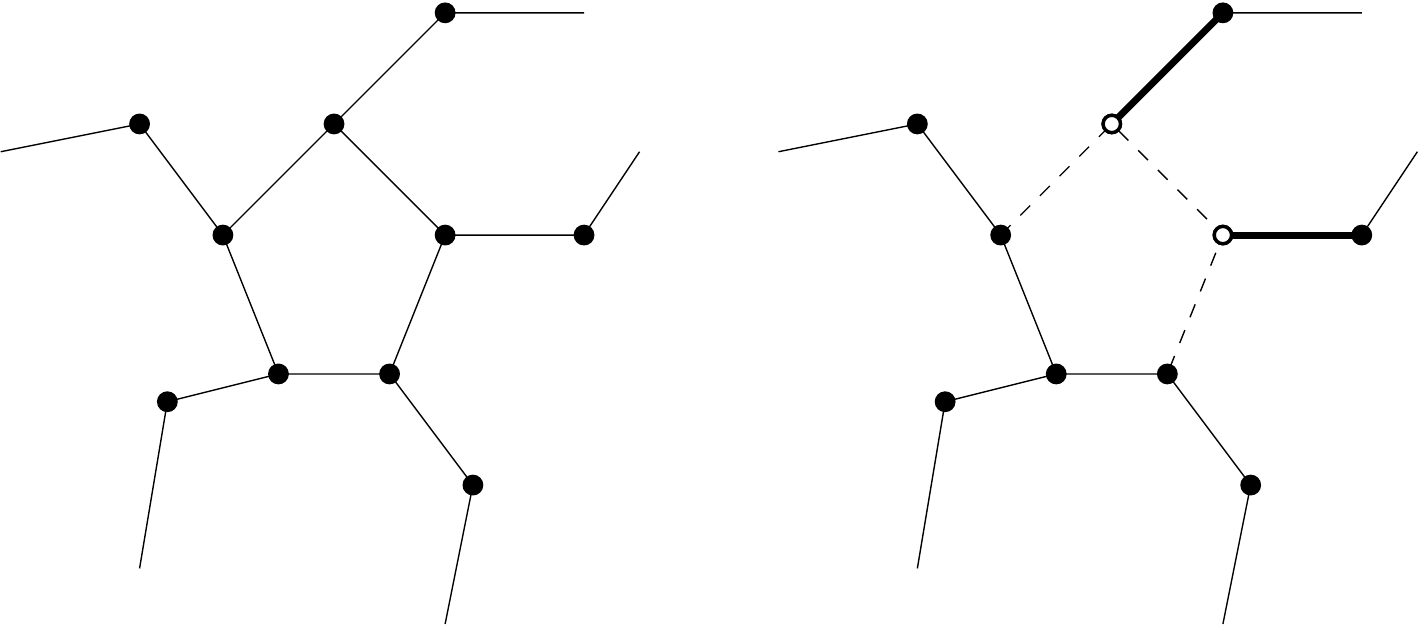}
 \caption{Removing three consecutive edges from each odd cycle in $G\in\mathcal G_{n,d}$ with length $\leq 16d - 24$. Thick edges are forced to be in the cover. Empty circles are forced to be good vertices.}
\label{fig:bc}
\end{figure}

Now we prove the following theorem.
\begin{theorem}\label{thm:last}
For every fixed $d\geq 3$, $G\in\mathcal{G}_{n,d}$ is a.a.s.\ normal.
\end{theorem}
\begin{proof}
First we delete three consecutive edges from each odd cycle of $G\in \mathcal{G}_{n,d}$ with length $\leq 16d - 19$ (see Figure \ref{fig:bc}) and then one edge from each 4-cycle. Now note that the expected number of copies of any subgraph $H$ of $G$ with $v$ vertices and $e$ edges is $O\left(n^{v-e}\right)$. Thus, the expected number of copies of subgraphs with more edges than vertices is $O\left(n^{-1}\right)$. Consequently, cycles of $G$ of length $\leq 16 d - 19$ are a.a.s.\ pairwise disjoint and pairwise at distance at least $32d$. Thus these cycles of $G$ are not connected with paths of bounded length a.a.s. (see \cite{janson2000random, wormald1999models}). Using this, along with the fact that $d\geq 3$, the resulting graph $G^\prime$ is a connected triangle-free graph a.a.s. Now we apply the algorithm proposed in the previous section to find a nice star cover $\mathcal C$ of $G^\prime$. We claim that the same star cover is nice for $G$. Adding back the deleted edges, every odd cycle of $G^\prime$ is nice. Consider odd cycles in $G$ using at least one of the edges in $E(G)\setminus E(G^{\prime})$. If such a cycle contains all three edges removed from a ``short" odd cycle, then two vertices on these edges will be good (see Figure \ref{fig:bc}). Since the number of good vertices in an odd cycle is odd due to Observation \ref{Obs:Obsr1}, we have at least three good vertices. If an odd cycle contains one or two edges removed from a ``short" odd cycle or a 4-cycle, using the fact that ``short" cycles are all disjoint and pairwise at distance at least $32d$, it is a.a.s.\ of length $\geq 32d$ and therefore  by Lemma \ref{lem:longest} it has at least three good vertices. Furthermore, in the proof of Theorem \ref{thm:MainGeneral} we showed that any odd cycle of length $>16d-19$ has at least 3 good vertices and consequently we obtain a cover for $G$ which is nice. Thus $G$ is a.a.s. normal.
\end{proof}


\bibliographystyle{abbrv}
\bibliography{normal}


\end{document}